\documentclass[reqno,11.5pt]{article}

\usepackage{diagrams}
\usepackage{amssymb}
\usepackage{tikz}
\usepackage{amsmath}
\usepackage{latexsym}
\usepackage{mathrsfs}
\usepackage{amsthm}
\usepackage{verbatim}
\usepackage{graphicx}
\usepackage{epstopdf}
\usepackage{epsfig}
\diagramstyle[labelstyle=\scriptstyle]
\usepackage{color}
\usepackage{float}
\usepackage[algo2e,linesnumbered,ruled]{algorithm2e}
\usepackage{titlesec}
\usepackage{bm}
\usepackage[colorlinks,linkcolor=blue,anchorcolor=green,citecolor=red]{hyperref}
\usepackage{amsfonts}
\usepackage{fancyhdr}
\usepackage{amscd}
\usepackage{cases}
\usepackage{amsmath,stackrel}

\usepackage{mathtools}

\newtheorem{theorem}{Theorem}

\newtheorem{definition}{Definition}

\newtheorem{lemma}{Lemma}
\newtheorem{problem}{Problem}

\newcommand{\0}{\mathaccent23}

\newcommand\tr{\operatorname{tr}}

\newcommand\grad{\operatorname{grad}}
\renewcommand\div{\operatorname{div}}
\newcommand\curl{\operatorname{curl}}

\newcommand{\bs}{{\scriptscriptstyle \bullet}}

\makeatletter
\@addtoreset{equation}{section}
\makeatother

\usepackage{geometry}
\begin{document}

\title{Generalized Gaffney inequality and discrete compactness for discrete differential forms}

\author{Juncai He\thanks{School of Mathematical Sciences, Peking University, Beijing 100871, China. email: {\tt juncaihe@pku.edu.cn}} \and Kaibo Hu\thanks{Department of Mathematics,
University of Oslo,Oslo 0316, Norway. email: {\tt kaibohu@math.uio.no}}\and Jinchao Xu\thanks{Department of Mathematics, Penn. State University, University Park, PA 16802, USA. Email: {\tt xu@math.psu.edu}}}

\date{}

\maketitle

\begin{abstract}
We prove generalized Gaffney inequalities and the discrete compactness for finite element differential forms on $s$-regular domains, including general Lipschitz domains. In computational electromagnetism, special cases of these results have been established for edge elements with weakly imposed divergence-free conditions and used in the analysis of nonlinear and eigenvalue problems. In this paper, we generalize these results to discrete differential forms,  not necessarily with strongly or weakly imposed constraints. The analysis relies on 
 a new Hodge mapping and its approximation property. As an application, we show $L^{p}$ estimates for several finite element approximations of the scalar and vector Laplacian problems.
\end{abstract}

\section{Introduction}

Let $\Omega\subset \mathbb{R}^{n}$ be an $s$-regular domain ($1/2\leq s\leq 1$) (c.f. \cite{mitrea2001layer}) with trivial cohomology and 
\begin{align*}
Z^{k}:&=\0H\Lambda^{k}\cap H^{\ast}\Lambda^{k}(\Omega)\\
&=\left \{ w\in L^{2}\Lambda^{k}(\Omega): \mathrm{d}w\in L^{2}\Lambda^{k+1}(\Omega), \left .\tr\right |_{\partial\Omega} w=0\right \}\cap \left \{ w\in L^{2}\Lambda^{k}(\Omega): \delta w\in L^{2}\Lambda^{k-1}(\Omega)\right \},
\end{align*}
be the space of differential $k$-forms with vanishing trace on the boundary. 
The generalized Gaffney inequality 
$$
\|w\|^{2}_{L^{p}}\leq C\left ( \|\mathrm{d}w\|_{L^{2}}^{2}+\|\delta w\|_{L^{2}}^{2}\right ),
$$
and the compactness $Z^{k}\hookrightarrow L^{p}(\Omega)$ are two important properties of $Z^{k}$ and play a crucial role  in the analysis of nonlinear and eigenvalue problems for differential forms (see, e.g., \cite{christiansen2011convergence,Schotzau.D.2004a,HLMZ18_2581,boffi2010finite}). 

For numerical methods for differential forms and Hodge Laplacian, approximation of $Z^{k}$ by the classical $C^{0}$ finite elements will cause notorious pseudo-solutions and instability (c.f. \cite{Hiptmair.R.2002a,Arnold.D;Falk.R;Winther.R.2006a,boffi2010finite}). To cure this problem, one could approximate $Z^{k}$ by a finite dimensional space $\0 H_{h}\Lambda^{k}\subset \0 H\Lambda^{k}$. We refer to \cite{Bossavit.A.1998a,Hiptmair.R.2002a,Arnold.D;Falk.R;Winther.R.2006a} for details on the discrete differential forms and the finite element exterior calculus. The space $\0 H_{h}\Lambda^{k}$ is a nonconforming approximation of $Z^{k}$ since the codifferential operator cannot be taken in the $L^{2}$ sense, and this causes a difficulty in the numerical analysis. In particular, the generalized Gaffney inequality and the compactness cannot be inherited from $Z^{k}$.

The discrete differential forms fit in a complex
\begin{equation}\label{complex-Vh}
\begin{diagram}
0  & \rTo^{}& \0 H_{h}\Lambda^{0} & \rTo^{\mathrm{d}} &\0H_{h}\Lambda^{1} & \rTo^{\mathrm{d}} & \cdots &  \rTo^{\mathrm{d}}&\0 H_{h}\Lambda^{n}  & \rTo^{}& 0.
\end{diagram}
\end{equation}
A discrete Hodge decomposition follows:
$$
\0H_{h}\Lambda^{k}= \mathrm{d}\0H_{h}\Lambda^{k-1}\oplus \left [ \mathrm{d}\0H_{h}\Lambda^{k-1}\right ]^{\perp}=\mathrm{d}\0H_{h}\Lambda^{k-1}\oplus \mathrm{d}_{h}^{\ast}\0H_{h}\Lambda^{k+1},
$$
where $\mathrm{d}_{h}^{\ast}$ is the $L^{2}$ adjoint operator of $\mathrm{d}: \0H_{h}\Lambda^{k}\mapsto \0H_{h}\Lambda^{k+1}$.

In computational electromagnetism, the electromagnetic fields are usually discretized in the discrete divergence-free edge element space $\left [\mathrm{d}\0H_{h}\Lambda^{0}\right ]^{\perp}$, i.e., in
\begin{align}\label{def:xhc}
{X_h^c}:=\left \{\bm{w}_{h}\in \0H_{h}\Lambda^{1}, (w_{h}, \grad\phi_{h})=0, ~\forall \phi_{h}\in \0H_{h}\Lambda^{0}\right  \},
\end{align}
where $\0H_{h}\Lambda^{1}$ is the N\'{e}d\'{e}lec edge element (first or second kind) and $\0H_{h}\Lambda^{0}$ is the Lagrange finite element with a suitable degree \cite{christiansen2011convergence,Hiptmair.R.2002a,Schotzau.D.2004a,HLMZ18_2581,boffi2010finite}.  The discrete divergence-free condition in $X_{h}^{c}$ reflects the Gauss laws in the Maxwell equations.  In this special case,  generalized Gaffney inequalities have been established in, e.g., \cite{christiansen2011convergence,Schotzau.D.2004a,HLMZ18_2581} for discretizations of nonlinear problems.
For eigenvalue problems,  the discrete compactness of ${X_h^c}$ is established and used for the convergence theory (see, e.g., \cite{Hiptmair.R.2002a,boffi2010finite} and the references therein). 
 The analysis of both the generalized Gaffney inequality and the discrete compactness is based on a  map $\mathcal{H}: X_{h}^{c} \mapsto H_{0}(\curl)\cap H(\div0)$ and its approximation property \cite{Hiptmair.R.2002a}.  This continuous lifting acts as a connection between the discrete and continuous levels and is sometimes referred to as the Hodge mapping.

For some problems in electromagnetism, the divergence-free constraint in $H_{0}(\curl)\cap H(\div0)$ plays a crucial role. Therefore strongly divergence-free Brezzi-Douglas-Marini or Raviart-Thomas finite elements $H^{h}(\div0)$ could be used to approximate electromagnetic fields, see \cite{hu2014stable}. To show the well-posedness of the finite element schemes, a new Hodge mapping is studied in \cite{hu2015structure}, see also  \cite{hu2017magnetic} for another type of boundary conditions. To the best of our knowledge, discrete compactness has not been discussed for $H^{h}(\div0)$.

The purpose of this paper is to prove the generalized Gaffney inequality and the discrete compactness for discrete differential forms on  $s$-regular domains.  This goal is achieved by defining a Hodge mapping for the entire discrete space without (either strong or weak) constraints such as the divergence-free conditions. This new Hodge mapping is a generalization of the classical technique for $X_{h}^{c}$ \cite{Hiptmair.R.2002a} and the result for $H^{h}(\div0)$ \cite{hu2015structure}. 
As we shall see, the results rely on the regularity at the continuous level and the existence of bounded cochain projections \cite{falk2014local}.

The rest of this paper  is organized as follows. In Section \ref{sec:preliminary}, we introduce some notation and preliminary results.  In Section \ref{sec:discreteSIDC} we show the main results. Detailed proofs, including the new Hodge mapping, are postponed to Section \ref{sec:Hodge}. In Section \ref{sec:vector}, we show some applications in the $L^{p}$ estimates of the Hodge Laplacian problems. In Section \ref{sec:conclusion} we give concluding remarks.

\section{Preliminaries}\label{sec:preliminary}

We introduce some notation and preliminary results. For differential forms and exterior derivatives, we follow the convention in
 \cite{Arnold.D;Falk.R;Winther.R.2006a} and refer to \cite{Arnold.D;Falk.R;Winther.R.2006a,lang2012fundamentals} for more details.

 We use $\Lambda^{k}(\Omega)$ to denote the space of smooth differential $k$-forms on $\Omega$.  Let $\star: \Lambda^{k}\mapsto \Lambda^{n-k}$ be the Hodge star operator.  We use $(\cdot, \cdot)$ to denote the $L^{2}$ inner product of $k$-forms (for any nonnegative integer $k$):
 $$
 (u, v):=\int_{\Omega}u\wedge \star v, \quad \forall u, v\in \Lambda^{k}(\Omega).
 $$
We denote the norm by
 $$
 \|u\|^{2}:=(u, u).
 $$
Define the Sobolev spaces of  differential $k$-forms:
$$
L^{2}\Lambda^{k}(\Omega):=\left \{v\in \Lambda^{k}(\Omega): (v, v) <\infty\right \},
$$
and
$$
H\Lambda^{k}(\Omega):=\left \{u\in L^{2}\Lambda^{k}(\Omega): ~\mathrm{d}u\in  L^{2}\Lambda^{k+1}(\Omega)\right \},
$$
where $\mathrm{d}$ is the exterior derivative. Define the $H\Lambda$ inner product and the corresponding norms:
$$
(u, v)_{H\Lambda}:=(u, v)+(\mathrm{d}u, \mathrm{d}v), \quad \|u\|_{H\Lambda}^{2}:=(u, u)_{H\Lambda}.
$$

We use $H^{s}\Lambda^{k}(\Omega)$ and $L^{p}\Lambda^{k}(\Omega)$ to denote the  $H^{s}$ and $L^{p}$ Sobolev spaces of differential forms where $s$ is a positive real number and $1\leq p\leq \infty$ is a positive integer  (c.f. \cite{Arnold.D;Falk.R;Winther.R.2006a}).  The corresponding norms are denoted by $\|\cdot\|_{s}$ and $\|\cdot\|_{0, p}$ respectively.
For $s=0$, we also use $\|\cdot\|_{0}$ to denote the $L^{2}$ norm $\|\cdot\|$.

 The codifferential operator $\delta_{k}: C^{\infty}\Lambda^{k}(\Omega)\mapsto C^{\infty}\Lambda^{k-1}(\Omega)$ is defined by $\star\delta_{k}=(-1)^{k} \mathrm{d}\star$. When there is no possible confusion, we omit the subscript and write $\delta$ for any $k$-form. We similarly define
$$
H^{\ast}\Lambda^{k}(\Omega):=\left \{u\in L^{2}\Lambda^{k}(\Omega): ~\delta u\in  L^{2}\Lambda^{k-1}(\Omega)\right \}.
$$
Define the norm
$$
\|{w}\|_{Z}^{2}:=\|{w}\|^{2}+\left \|\mathrm{d}{w}\right \|^{2}+\left \|\delta {w}\right \|^{2}, \quad\forall {w}\in Z^{k}.
$$

We use the notation $u\lesssim v$ to denote $u\leq Cv$, where $C$ is a generic positive constant.

For  $0\leq s\leq 1$, a domain $\Omega$ is called {\it $s$-regular}, if   for any $z\in {Z}^{k}(\Omega)$, the following estimate holds:
\begin{align}\label{s-regularity}
\|z\|_{s}^{2}\lesssim \|\mathrm{d}z\|^{2}+\|\delta z\|^{2}.
\end{align}
We refer to \cite{Arnold.D;Falk.R;Winther.R.2006a,jakab2009on} with the references therein for more details on  $s$-regular domains in $\mathbb{R}^n$ and \cite{mitrea2001layer} for manifolds. Particularly, any Lipschitz domain is an $s$-regular domain for $s\geq 1/2$ \cite{mitrea2001layer}. For any polyhedron in $\mathbb{R}^{3}$ we can choose $s\in (1/2, 1]$ \cite{amrouche1998vector} and for convex domains we can choose $s=1$. 

We assume that $\Omega$ is an $s$-regular domain. For ease of presentation, we further assume that all Betti numbers except for the zeroth vanish, meaning that the de Rham complex on $\Omega$ has trivial cohomology. Therefore there are no nontrivial harmonic forms.  


Let $\mathrm{tr}$ be the trace operator.
We use  $\0 H\Lambda^{k}(\Omega), 0\leq k\leq n-1$ to denote the space of differential $k$-forms with vanishing traces on $\partial \Omega$.  For $n$-forms in $n$ space dimensions, we formally define
$$
\0H\Lambda^{n}(\Omega):=\left \{q\in H\Lambda^{n}(\Omega): ~\int_{\Omega}q=0\right \}.
$$
 We also define
$$
\0 H^{\ast}\Lambda^{k}(\Omega):=\left \{u\in H^{\ast}\Lambda^{k}(\Omega): \mathrm{tr}\star u=0 \right \},\quad 1\leq k\leq n,
$$
$$
\0 H^{\ast}\Lambda^{0}(\Omega):=\left \{u\in H^{\ast}\Lambda^{0}(\Omega): ~\int_{\Omega}\star u=0\right  \},
$$
and define the spaces with vanishing exterior derivatives and coderivatives:
$$
H\Lambda^{k}(0, \Omega):=\left \{u\in H\Lambda^{k}(\Omega): \mathrm{d} u=0 \right \}, \quad \mbox{ and  } \quad
H^{\ast}\Lambda^{k}(0, \Omega):=\left \{u\in H^{\ast}\Lambda^{k}(\Omega):  \delta u=0 \right \}.
$$

The de Rham complex
\begin{equation}\label{deRham}
\begin{diagram}
0& \rTo & \mathbb{R} & \rTo^{\subset} &H \Lambda^{0}(\Omega) & \rTo^{\mathrm{d}} &H \Lambda^{1}(\Omega) & \rTo^{\mathrm{d}} & \cdots &  \rTo^{\mathrm{d}}&H \Lambda^{n}(\Omega)  & \rTo^{}& 0,
\end{diagram}
\end{equation}
is exact on  $\Omega$ with trivial cohomology, i.e. for any $u\in H\Lambda^{k}(\Omega)$ satisfying $\mathrm{d}u=0$, there exists $w\in H\Lambda^{k-1}(\Omega)$ such that $u=\mathrm{d}w$.  Similarly, the spaces with vanishing traces
\begin{equation}\label{deRham0}
\begin{diagram}
0 & \rTo^{} &\0H \Lambda^{0}(\Omega) & \rTo^{\mathrm{d}} & \0 H\Lambda^{1}(\Omega) & \rTo^{\mathrm{d}} & \cdots &  \rTo^{\mathrm{d}}& \0 H\Lambda^{n}(\Omega)  & \rTo^{}& 0,
\end{diagram}
\end{equation}
and the $L^{2}$ dual complex of  \eqref{deRham}
\begin{equation}\label{deRham0-delta}
\begin{diagram}
0 & \lTo^{} & \0H^{\ast}\Lambda^{0}(\Omega) & \lTo^{\delta} & \0H^{\ast}\Lambda^{1}(\Omega) & \lTo^{\delta} & \cdots &  \lTo^{\delta}& \0H^{\ast}\Lambda^{n}(\Omega)  & \lTo^{}& 0,
\end{diagram}
\end{equation}
are also exact sequences.

We assume that the sequence
\begin{equation}\label{complex-Vh}
\begin{diagram}
0& \rTo & \mathbb{R} & \rTo^{\subset} & H_{h}\Lambda^{0} & \rTo^{\mathrm{d}} &H_{h}\Lambda^{1} & \rTo^{\mathrm{d}} & \cdots &  \rTo^{\mathrm{d}}& H_{h}\Lambda^{n}  & \rTo^{}& 0,
\end{diagram}
\end{equation}
and the sequence with vanishing traces:
\begin{equation}\label{complex-Vh0}
\begin{diagram}
0& \rTo & \0H_{h}\Lambda^{0} & \rTo^{\mathrm{d}} &\0H_{h}\Lambda^{1} & \rTo^{\mathrm{d}} & \cdots &  \rTo^{\mathrm{d}}& H_{h}\Lambda^{n}/\mathbb{R}  & \rTo^{}& 0,
\end{diagram}
\end{equation}
are subcomplexes of \eqref{deRham}, i.e., $\0H_{h}\Lambda^{k}\subset H_{h}\Lambda^{k}\subset H\Lambda^{k}(\Omega)$, $\forall \, 0\leq k\leq n$, and each space has finite dimensions. Here $H_{h}\Lambda^{n}/\mathbb{R} $, also denoted as $\0 H_{h}\Lambda^{n}$, is the space of the discrete $n$-forms with vanishing integral. 
 Examples of \eqref{complex-Vh} include the finite element spaces in the Finite Element Periodic Table \cite{arnold2014periodic} with suitable order, e.g. the Lagrange $H^{1}$ elements, the 1st or the 2nd N\'{e}d\'{e}lec $H(\curl)$ elements and the Raviart-Thomas or the Brezzi-Douglas-Marini $H(\div)$ elements. For these finite elements, the  existence of Fortin operators implies that both \eqref{complex-Vh} and \eqref{complex-Vh0} are exact on domains with trivial cohomology.

 We use $\Pi_{k}: \0H\Lambda^{k}(\Omega)\mapsto \0H_{h}\Lambda^{k}$ to denote the interpolation operator for $k$-forms.  The construction of the interpolation operators for the finite element de Rham complexes can be found in e.g., \cite{falk2014local,christiansen2011topics,Arnold.D;Falk.R;Winther.R.2006a,schoberl2001commuting}. These interpolations commute with the exterior derivatives, i.e. $\mathrm{d}_{k}\Pi_{k}=\Pi_{k+1}\mathrm{d}_{k}$, where $\mathrm{d}_{k}$ is the exterior derivative for $k$-forms. Moreover, these operators are bounded with respect to both $L^{2}$ and $H\Lambda^{k}$ norms.  Below we assume that the interpolations $\Pi_{k}$, $k=0, 1, \cdots, n,$ are $L^{p}$-$L^{p}$ bounded \cite{christiansen2011topics,ern2016mollification}, i.e., there exists a generic positive constant $C$ such that
$$
\left \|\Pi_{k}{u}\right \|_{0, p}\leq C\|{u}\|_{0, p}, \quad \forall {u}\in L^{p}\Lambda^{k}(\Omega)\cap\0 H\Lambda^{k}(\Omega).
$$ The commutativity and the boundedness will be crucial in the sequel.

For \eqref{complex-Vh}, we define $\mathrm{d}_{h}^{\ast}: H_{h}\Lambda^{k}\mapsto H_{h}\Lambda^{k-1}$ as the $L^{2}$ dual of the exterior derivatives in \eqref{complex-Vh}, i.e., for any nonnegative integer $k$,
\begin{align}\label{def:Dstar}
(\mathrm{d}_{h}^{\ast}u_{h}, v_{h})=(u_{h}, \mathrm{d}v_{h}), \quad \forall v_{h}\in  H_h\Lambda^{k-1} .
\end{align}
Correspondingly, for \eqref{complex-Vh0} we define $\mathrm{d}_{h}^{\ast}: \0H_{h}\Lambda^{k}\mapsto \0H_{h}\Lambda^{k-1}$ by
\begin{align}\label{def:Dstar0}
(\mathrm{d}_{h}^{\ast}u_{h}, v_{h})=(u_{h}, \mathrm{d}v_{h}), \quad \forall v_{h}\in  \0H_h\Lambda^{k-1} .
\end{align}
Since $(\cdot, \cdot)$ is a complete inner product on finite dimensional spaces, the identity \eqref{def:Dstar} or \eqref{def:Dstar0} uniquely defines $\mathrm{d}_{h}^{\ast}$.
By definition we  have (for either \eqref{complex-Vh} or \eqref{complex-Vh0})
$$
\left ( \mathrm{d}_{h}^{\ast}\mathrm{d}_{h}^{\ast}u_{h}, v_{h}\right )=\left ( \mathrm{d}_{h}^{\ast}u_{h}, \mathrm{d}v_{h}\right )=\left (u_{h},  \mathrm{d}\mathrm{d}v_{h}\right )=0, \quad\forall u_{h}, v_{h}.
$$
Therefore we have               
$$
\left ( \mathrm{d}_{h}^{\ast}\right )^{2}=0,
$$
which mimics the identity $\delta^{2}=0$ at the continuous level. In this way we obtain the complexes
\begin{equation*}
\begin{diagram}
0& \lTo & \mathbb{R} & \lTo^{} & H_{h}\Lambda^{0} & \lTo^{\mathrm{d}_{h}^{\ast}} &H_{h}\Lambda^{1} & \lTo^{\mathrm{d}_{h}^{\ast}} & \cdots &  \lTo^{\mathrm{d}_{h}^{\ast}}& H_{h}\Lambda^{n}  & \lTo^{}& 0,
\end{diagram}
\end{equation*}
and
\begin{equation*}
\begin{diagram}
0& \lTo & \0H_{h}\Lambda^{0} &\lTo^{\mathrm{d}_{h}^{\ast}} &\0H_{h}\Lambda^{1} & \lTo^{\mathrm{d}_{h}^{\ast}} & \cdots &  \lTo^{\mathrm{d}_{h}^{\ast}}& H_{h}\Lambda^{n}/\mathbb{R}  & \rTo^{}& 0.
\end{diagram}
\end{equation*}

We define the range
$$
\0\mathfrak{B}_{h}^{k}:=\mathrm{d}\0H_{h}\Lambda^{k-1}(\Omega).
$$
Since we assume that $\Omega$ has trivial cohomology, the range is identical to the kernel space
$$\0\mathfrak{B}_{h}^{k}=\0\mathfrak{Z}_{h}^{k}:=\{u_{h}\in  \0H_h\Lambda^{k} : \mathrm{d}u_{h}=0\}.$$

For the discrete $L^{2}$ adjoint operators, we define $\0\mathfrak{B}^{\ast}_{k, h}:=\mathrm{d}_{h}^{\ast}\0H_{h}\Lambda^{k+1}$. 
For $u_{h}\in \0\mathfrak{B}^{\ast}_{k, h}$ and $w_{h}\in \0\mathfrak{Z}_{h}^{k}$, we have
$$
(u_{h}, w_{h})=(\mathrm{d}_{h}^{\ast}\phi_{h}, w_{h})=(\phi_{h}, \mathrm{d}w_{h})=0.
$$
Therefore $ \0\mathfrak{B}^{\ast}_{k, h}\perp \0 \mathfrak{Z}_{h}^{k}$. The orthogonality can be understood either with respect to the inner product $(\cdot, \cdot)$ or  with respect to $(\cdot, \cdot)_{H\Lambda}$.

The discrete Hodge decomposition holds:
\begin{align}\label{discrete-Hodge-1}
 H_h\Lambda^{k} ={\mathfrak{B}}_{h}^{k} \oplus \mathfrak{B}^{\ast}_{k, h}.
\end{align}
Analogously, we can decompose $\0 H_{h}\Lambda^{k}$  with vanishing boundary conditions:
\begin{align}\label{discrete-Hodge-2}
\0 H_{h}\Lambda^{k}=\mathrm{d}\0 H_{h}\Lambda^{k-1}\oplus \mathrm{d}_{h}^{\ast}\0 H_{h}^{\ast}\Lambda^{k+1}.
\end{align}

\section{Main results}\label{sec:discreteSIDC}

The generalized Gaffney inequality and the discrete compactness below are based on a key result:
\begin{lemma}[generalized Hodge mapping]\label{lem:hodge-mapping}
Let $\Omega$ be an $s$-regular domain. There exists a map $\mathcal{H}^{k}: \0 H_{h}\Lambda^{k}(\Omega)\mapsto Z^{k}$ such that 
$$
\|{u}_{h}-{\mathcal{H}^k}{u}_{h}\|\lesssim h^{s}\left (\|\mathrm{d}{u}_{h}\|+  \|\mathrm{d}_{h}^{\ast}{u}_{h}\|   \right ), \quad \forall{u}_{h}\in \0H_{h}\Lambda^{k}.
$$
\end{lemma}
We postpone the proof of this technical result to Section \ref{sec:Hodge}.

Based on Lemma \ref{lem:hodge-mapping}, we establish the generalized Gaffney inequality.
\begin{theorem}[generalized Gaffney inequality]\label{thm:discrete-sobolev}
Assume that $\Omega$ is an $s$-regular domain.  We have
$$
\|{u}_{h}\|_{0,p}\lesssim \|\mathrm{d}{u}_h\|+ \|\mathrm{d}_{h}^{\ast}{u}_{h}\|,\quad \forall {u}_{h}\in \0 H_{h}\Lambda^{k}(\Omega),
$$
where $p=2n/(n-2s)$ and $n$ is the space dimension.
\end{theorem}
For $n=3, s=1/2$, we have $p=3$ and for $n=3, s=1$, we have $p=6$.
\begin{proof}
From the triangular inequality,  we have
$$
\|{u}_{h}\|_{0,p}\leq \|{u}_{h}-\Pi_{k}{\mathcal{H}^k}{{u}_{h}}\|_{0,p}+\|\Pi_{k}{\mathcal{H}^k}{{u}_{h}}\|_{0,p}.
$$
From the inverse estimates, the interpolation error estimates and the approximation of the Hodge mapping (Lemma \ref{lem:hodge-mapping}),
\begin{align*}
\|{u}_{h}-\Pi_{k}{\mathcal{H}^k}{{u}_{h}}\|_{0,p}&\lesssim h^{-\left ( \frac{n}{2}-\frac{n}{p}\right )}\|{u}_{h}-\Pi_{k}{\mathcal{H}^k}{{u}_{h}}\|\\&
\lesssim h^{-\left ( \frac{n}{2}-\frac{n}{p}\right )}( \|{u}_{h}-{\mathcal{H}^k}{{u}_{h}}\|+\|{\mathcal{H}^k}{{u}}-\Pi_{k}{\mathcal{H}^k}{{u}_{h}}\|   )\\&
\lesssim h^{-\left ( \frac{n}{2}-\frac{n}{p}\right )}h^{s}\left ( \|\mathrm{d}{u}_{h}\|+\|\mathrm{d}_{h}^{\ast}{u}_{h}\|\right )\\&
\lesssim \|\mathrm{d}{u}_{h}\|+\|\mathrm{d}_{h}^{\ast}{u}_{h}\|.
\end{align*}
From the $L^{p}$ boundedness of the interpolation operators and the regularity of ${Z}^{k}$, we have
$$
\|\Pi_{k}{\mathcal{H}^k}{{u}_{h}}\|_{0, p}\lesssim \|{\mathcal{H}^k}{{u}_{h}}\|_{0, p}\lesssim \|\mathrm{d}{\mathcal{H}^k}{u}_{h}\| +\|\delta {\mathcal{H}^k}{u}_{h}\|\|\leq \|\mathrm{d}{u}_{h}\|+ \|\mathrm{d}_{h}^{\ast} {u}_{h}\|.
$$
This completes the proof.
\end{proof}

Let $\mathscr{H}  = \{h_n : n =1,2, \cdots\}$ be a sequence of decreasing positive real numbers converging to zero and $\left \{\mathcal{T}_{h}\right \}_{h\in \mathscr{H} }$ be a family of shape-regular meshes on $\Omega$.
\begin{theorem}[discrete compactness]\label{thm:discrete-compactness}
Given a sequence ${u}_{h}\in \0H_{h}\Lambda^{k}( \Omega), ~h\in \mathscr{H}$ satisfying $\|\mathrm{d} {u}_{h}\|+\|\mathrm{d}_{h}^{\ast} {u}_{h}\|\leq C$, where $C$ is a positive constant, there exists a subsequence ${u}_{h_{n}}$ which converges strongly in $L^{2}\Lambda^{k}(\Omega)$.
\end{theorem}
\begin{proof}
From the regularity of ${Z}^{k}$  \eqref{s-regularity} and the definition of ${\mathcal{H}^k}$,  we have
$$
\|{\mathcal{H}^k}{{u}_{h}}\|_{Z}\lesssim \|\mathrm{d} {\mathcal{H}^k}{u}_{h}\|+\|\delta {\mathcal{H}^k}{u}_{h}\|\leq \|\mathrm{d}{u}_{h}\|+\|\mathrm{d}_{h}^{\ast} {u}_{h}\|\leq C.
$$
Since ${Z}^{k}$ is compactly imbedded in $L^{2}(\Omega)$,  there exists a sequence converging strongly in $L^{2}\Lambda^{k}(\Omega)$:
\begin{align}\label{convergence-B0}
{\mathcal{H}^k}{u}_{h_{n}}\rightarrow {{u}}_{0}, \quad \mbox{ as }n\rightarrow \infty.
\end{align}
Next we prove ${u}_{h_{n}}\rightarrow {{u}}_{0}$ strongly in $L^{2}(\Omega)$. In fact, from the triangular inequality:
\begin{align}
\|{u}_{h_{n}}-{{u}}_{0}\|\leq \|{u}_{h_{n}}-{\mathcal{H}^k}{u}_{h_{n}}\|+\|{\mathcal{H}^k}{u}_{h_{n}}-{u}_{0}\|.
\end{align}
Due to the approximation property of the Hodge mapping (Lemma \ref{Hd-approximation}),
$$
\|{u}_{h_{n}}-{\mathcal{H}^k}{u}_{h_{n}}\|\lesssim h^{s}\left ( \|\mathrm{d}{u}_{h_{n}}\|+\|\mathrm{d}_{h}^{\ast} {u}_{h_{n}}\| \right )\lesssim h^{s}\rightarrow 0,
$$
as $n\rightarrow \infty$ (and hence $h_{n}\rightarrow 0$).

Due to \eqref{convergence-B0},
$$
\|{\mathcal{H}^k}{u}_{h_{n}}-{{u}}_{0}\|\rightarrow 0.
$$
This completes the proof.
\end{proof}

Theorem \ref{thm:discrete-sobolev} and Theorem \ref{thm:discrete-compactness} are based on the complexes \eqref{deRham0} and \eqref{complex-Vh0} with vanishing boundary conditions. The same conclusions in Theorem \ref{thm:discrete-sobolev} and Theorem \ref{thm:discrete-compactness}  hold for $H_{h}\Lambda^{k}$ and the proof can be translated verbatim in this case by using the complexes \eqref{deRham} and \eqref{complex-Vh}.

\section{Generalized Hodge mapping}\label{sec:Hodge}

This section is devoted to the proof of Lemma \ref{lem:hodge-mapping}. The proof consists of two steps: first, we generalize the classical Hodge mapping for the edge elements to discrete differential forms with weak constraints ($u_{h}\in \0H_{h}\Lambda^{k}(\Omega)$ satisfying $\mathrm{d}_{h}^{\ast}u_{h}=0$); second, we define a generalized Hodge mapping for the entire space $\0H_{h}\Lambda^{k}(\Omega)$ and prove its approximation properties.

\paragraph{Hodge mapping for weakly constrained spaces}

Let ${Z}_{0}^{k}:=\0 H\Lambda^{k}( \Omega)\cap {H}^{\ast}\Lambda^{k}(0, \Omega)$ be the subspace of $Z^{k}$ with vanishing codifferential.

For discrete differential forms, we define a Hodge mapping  $\mathcal{H}_{0}^{k}: \0\mathfrak{B}_{k, h}^{\ast} \mapsto {Z}_{0}^{k}$:
$$
\mathrm{d}\mathcal{H}_{0}^{k}\phi_{h}=\mathrm{d}\phi_{h}, \quad\forall \phi_{h}\in\0\mathfrak{B}_{k, h}^{\ast}.
$$
The Poincar\'{e} inequality in $Z^{k}_{0}$ (\eqref{s-regularity} with $s=0$) implies that ${\mathcal{H}_{0}^{k}}$ is well-defined. Here ${\mathcal{H}_{0}^{k}}$ is a generalization of the Hodge mapping for the weakly divergence-free edge elements $X_{h}^{c}$ \cite{Hiptmair.R.2002a}.

We then show the approximation property of ${\mathcal{H}^k_{0}}$. The proof is a generalization of the properties of the Hodge mapping for $X_{h}^{c}$ (c.f. \cite[Lemma 4.5]{Hiptmair.R.2002a}).
\begin{theorem}\label{thm:part-hodge}
Assume that $\Omega$ is an  $s$-regular domain where $s\in [1/2, 1]$.  We have
\begin{align}
\|{u}_{h}-{\mathcal{H}^k_{0}}{u}_{h}\|\lesssim h^{s}\|\mathrm{d}{u}_{h}\|, \quad \forall {u}_{h}\in \0 \mathfrak{B}_{k, h}^{\ast}.
\end{align}
\end{theorem}
\begin{proof}
We have
$$
\|{u}_{h}-{\mathcal{H}_{0}^{k}}{u}_{h}\|\leq \|{u}_{h}-\Pi_{k}{\mathcal{H}_{0}^{k}}{u}_{h}\|+\|\Pi_{k}{\mathcal{H}_{0}^{k}}{u}_{h}-{\mathcal{H}_{0}^{k}}{u}_{h}\|.
$$
For the first term,
\begin{align*}
\|{u}_{h}-\Pi_{k}{\mathcal{H}_{0}^{k}}{u}_{h}\|^{2}&=\left ( {u}_{h}-\Pi_{k}{\mathcal{H}_{0}^{k}}{u}_{h}, {u}_{h}-\Pi_{k}{\mathcal{H}_{0}^{k}}{u}_{h}   \right )\\
&=({u}_{h}-\Pi_{k}{\mathcal{H}_{0}^{k}}{u}_{h}, {u}_{h}-{\mathcal{H}_{0}^{k}}{u}_{h})+({u}_{h}-\Pi_{k}{\mathcal{H}_{0}^{k}}{u}_{h}, {\mathcal{H}_{0}^{k}}{u}_{h}-\Pi_{k}{\mathcal{H}_{0}^{k}}{u}_{h}).
\end{align*}
We note that
$$
\mathrm{d}({u}_{h}-\Pi_{k}{\mathcal{H}_{0}^{k}}{u}_{h})=0,
$$
due to the commuting diagram and the definition of ${\mathcal{H}_{0}^{k}}$. Therefore there exists ${\phi}_{h}\in \0 H_{h}\Lambda^{k-1}(\Omega)$, such that
$$
{u}_{h}-\Pi_{k}{\mathcal{H}_{0}^{k}}{u}_{h}=\mathrm{d}{\phi}_{h}.
$$
This implies
\begin{align*}
({u}_{h}-\Pi_{k}{\mathcal{H}_{0}^{k}}{u}_{h}, {u}_{h}-{\mathcal{H}_{0}^{k}}{u}_{h})=(\mathrm{d}{\phi}_{h}, {u}_{h}-{\mathcal{H}_{0}^{k}}{u}_{h})=({\phi}_{h}, \mathrm{d}_{h}^{\ast}{u}_{h}-\delta{\mathcal{H}_{0}^{k}}{u}_{h})=0.
\end{align*}
Consequently,
\begin{align*}
\|{u}_{h}-{\mathcal{H}_{0}^{k}}{u}_{h}\|\lesssim \|\Pi_{k}{\mathcal{H}_{0}^{k}}{u}_{h}- {\mathcal{H}_{0}^{k}}{u}_{h}\|\lesssim h^{s}\|{\mathcal{H}_{0}^{k}}{u}_{h}\|_{s}\lesssim h^{s}\|\mathrm{d}{\mathcal{H}_{0}^{k}}{u}_{h}\|= h^{s}\|\mathrm{d}{u}_{h}\|.
\end{align*}
Here the second inequality follows from the estimates for the interpolation operators \cite{Arnold.D;Falk.R;Winther.R.2006a}.
\end{proof}

\paragraph{Hodge mapping for the entire space}

We first prove a discrete Poincar\'{e} inequality for the entire space $\0H_{h}\Lambda^{k}(\Omega)$. Special cases of Theorem \ref{thm:discrete-poincare} for the face elements $H^{h}_{0}(\div, \Omega)$ and the edge elements $H^{h}_0(\curl, \Omega)$ can be found in \cite{chen2016multigrid} and \cite{hu2015structure}.
\begin{theorem}[discrete Poincar\'{e} inequality]\label{thm:discrete-poincare}
There exists a generic positive constant $C$ such that
\begin{align}\label{discrete-poincare-1}
\|{u}_{h}\|^{2}\leq C\left ( \| \mathrm{d}{u}_{h}\|^{2}+\|\mathrm{d}_{h}^{\ast}{u}_{h}\|^{2}\right ), \quad \forall {u}_{h}\in \0 H_{h}\Lambda^{k}(\Omega).
\end{align}
\end{theorem}
\begin{proof}
 For any  ${u}_{h}\in \0 H_{h}\Lambda^{k}(\Omega)$, we have the Hodge decomposition ${u}_{h}={u}_{1}+{u}_{2}$, where ${u}_{1}\in \left (\0{\mathfrak{B}}^{k}_{h}\right )^{\perp}$ satisfies $\mathrm{d}_{h}^{\ast}{u}_{1}=0$ and  ${u}_{2}\in \0{\mathfrak{B}}^{k}_{h}$ satisfies $\mathrm{d}{u}_{2}=0$.  For ${u}_{1}$, we have $ \|u_{1}\|\leq C\|\mathrm{d}u_{1}\|=C\|\mathrm{d}u_{h}\|$ (c.f. \cite{Arnold.D;Falk.R;Winther.R.2006a}). Then it remains to show $\|u_{2}\|\leq C\|\mathrm{d}_{h}^{\ast}u_{2}\|= C\|\mathrm{d}_{h}^{\ast}u\|$.

In fact, for ${u}_{2}\in \0\mathfrak{B}_{h}^{k}$ we can choose $v_{h}\in \0 H_{h}\Lambda^{k-1}(\Omega)$ such that $\mathrm{d} {v}_{h}={u}_{2}$ and $\mathrm{d}_{h}^{\ast}{v}_{h}=0$. By the classical discrete Poincar\'e inequality in \cite{Arnold.D;Falk.R;Winther.R.2006a}, we have $\|{v}_{h}\|_{H\Lambda}\lesssim \|\mathrm{d}{v}_{h}\|=\|{u}_{2}\|$.

Then we have
\begin{align}\label{poincare-l2dual}
\|\mathrm{d}_{h}^{\ast}{u}_{2}\|=\sup_{{w}_{h}\in \0H_{h}\Lambda^{k-1}(\Omega)}\frac{(\mathrm{d}_{h}^{\ast}{u}_{2}, {w}_{h})}{\|{w}_{h}\|}=\sup_{{w}_{h}\in \0H_{h}\Lambda^{k-1}(\Omega)}\frac{({u}_{2}, \mathrm{d}{w}_{h})}{\|{w}_{h}\|}\geq \frac{({u}_{2}, \mathrm{d} {v}_{h})}{\|{v}_{h}\|}\gtrsim \|{u}_{2}\|.
\end{align}
\end{proof}

Now we are in a position to define a generalized Hodge mapping. 
Define ${\mathcal{H}^k}: \0 H_{h}\Lambda^{k}(\Omega)\mapsto Z^{k}$ by
\begin{equation}\label{def:Hk}
\begin{cases}
&\mathrm{d}{\mathcal{H}^k}{u}_{h}=\mathrm{d}{u}_{h}, \\
&\left (  \delta{\mathcal{H}^k}{u}_{h}, \delta z \right )=\left ( \mathrm{d}_{h}^{\ast}{u}_{h}, \delta z \right ), \quad \forall z\in Z^{k}.
\end{cases}
\end{equation}
Using the identity   $(\mathrm{d}{\mathcal{H}^k}{u}_{h}, \delta{\mathcal{H}^k}{u}_{h})=0$ and the Poincar\'{e} inequality in $Z^{k}$, i.e.,
$$
\|{\mathcal{H}^k}{u}_{h}\|\lesssim \|\mathrm{d}{\mathcal{H}^k}{u}_{h}\|+\|\delta {\mathcal{H}^k}{u}_{h}\|,
$$
we see that ${\mathcal{H}^k}$ is well-defined.

Taking $z={\mathcal{H}^k}{u}_{h}$ in \eqref{def:Hk}, we obtain
\begin{align}
\|\delta  {\mathcal{H}^k}{u}_{h}\|\leq \|\mathrm{d}_{h}^{\ast} {u}_{h}\|.
\end{align}

By the Hodge decomposition at the continuous level \cite{Arnold.D;Falk.R;Winther.R.2006a}, we have  $\delta \0H^{\ast}\Lambda^{k}(\Omega)=\delta {Z}^{k}$. Therefore taking $\delta z={w}\in \delta \0H^{\ast}\Lambda^{k-1}(\Omega)$ in \eqref{def:Hk}, we have
\begin{align}\label{d-ast-H}
\left(\delta {\mathcal{H}^k}{u}_{h}, {w} \right)=\left (\mathrm{d}_{h}^{\ast}{u}_{h}, {w} \right ), \quad \forall {w}\in \delta \0H^{\ast}\Lambda^{k}(\Omega).
\end{align}

Finally we prove the approximation of ${\mathcal{H}^k}$.
\begin{theorem}\label{Hd-approximation}
Let $\Omega$ be an $s$-regular domain. We have
$$
\|{u}_{h}-{\mathcal{H}^k}{u}_{h}\|\lesssim h^{s}\left (\|\mathrm{d}{u}_{h}\|+  \|\mathrm{d}_{h}^{\ast}{u}_{h}\|   \right ), \quad \forall{u}_{h}\in \0H_{h}\Lambda^{k}.
$$
\end{theorem}
\begin{proof}
Thanks to the commuting diagram (the interpolation operator $\Pi_{\bs}$ commutes with the exterior derivatives), we have $\mathrm{d}
  \left({u}_{h}-\Pi_{k}{\mathcal{H}^k}{{u}_{h}}\right)=0$, so  there exists ${\phi}_{h}\in
\0{\mathfrak{B}}_{k-1, h}^{\ast}$ satisfying $\mathrm{d}_{h}^{\ast}\phi_{h}=0$  such that
  ${u}_{h}-\Pi_{k}{\mathcal{H}^k}{{u}}=\mathrm{d}
  {\phi}_{h}=\mathrm{d} {\mathcal{H}_0^{k-1}}{{\phi}_{h}}$ and
  \begin{equation}
    \label{phi-phi}
\|{\phi}_{h}-{\mathcal{H}_0^{k-1}}{{\phi}_{h}}\|\lesssim h^{s}\|\mathrm{d} {\phi}_{h}\|=h^{s}\|{u}_{h}-\Pi_{k}{\mathcal{H}^k}{u}_{h}\|,
  \end{equation}

From the definition of ${\mathcal{H}^k}{{u}_{h}}$, we have
$$
(\mathrm{d}_{h}^{\ast}{u}_{h}, {\mathcal{H}_0^{k-1}}{{\phi}_{h}})=(\delta {\mathcal{H}^k}{{u}_{h}}, {\mathcal{H}_0^{k-1}}{{\phi}_{h}})=({\mathcal{H}^k}{{u}_{h}}, \mathrm{d}{\mathcal{H}_0^{k-1}}{{\phi}_{h}}),
$$
and
$$
({u}_{h}, \mathrm{d} {\phi}_{h})=(\mathrm{d}_{h}^{\ast} {u}_{h}, {\phi}_{h})=(\mathrm{d}_{h}^{\ast} {u}_{h}, {\phi}_{h}-{\mathcal{H}_0^{k-1}}{{\phi}_{h}})+({\mathcal{H}^k}{{u}_{h}}, \mathrm{d} {\mathcal{H}_0^{k-1}}{{\phi}_{h}}).
$$
The last identity is due to \eqref{d-ast-H}.
Therefore,
$$
({u}_{h}-{\mathcal{H}^k}{{u}_{h}}, {u}_{h}-\Pi_{k}{\mathcal{H}^k}{{u}_{h}})=(\mathrm{d}_{h}^{\ast} {u}_{h}, {\phi}_{h}-{\mathcal{H}_0^{k-1}}{{\phi}_{h}}).
$$
Thus
\begin{align*}
\|{u}_{h}-{\mathcal{H}^k}{{u}_{h}}\|^{2}&=({u}_{h}-{\mathcal{H}^k}{{u}_{h}}, {u}_{h}-\Pi_{k}{\mathcal{H}^k}{{u}_{h}})+({u}_{h}-{\mathcal{H}^k}{{u}_{h}}, \Pi_{k}{\mathcal{H}^k}{{u}_{h}}-{\mathcal{H}^k}{{u}_{h}})\\&
=(\mathrm{d}_{h}^{\ast} {u}_{h}, {\phi}_{h}-{\mathcal{H}_0^{k-1}}{{\phi}_{h}})+({u}_{h}-{\mathcal{H}^k}{{u}_{h}}, \Pi_{k}{\mathcal{H}^k}{{u}_{h}}-{\mathcal{H}^k}{{u}_{h}}).
\end{align*}

From Theorem \ref{thm:part-hodge},
$$
\|{\mathcal{H}^k}{{u}_{h}}-\Pi_{k}{\mathcal{H}^k}{{u}_{h}}\|\lesssim h^{s}\|{\mathcal{H}^k}{{u}_{h}}\|_{s}
\lesssim h^{s} \left (  \|\mathrm{d} {\mathcal{H}^k}{{u}_{h}}\| + \|\delta{\mathcal{H}^k}{u}_{h}\|\right )       \leq h^{s} \left (  \|\mathrm{d}{u}_{h}\|+\|\mathrm{d}_{h}^{\ast}{u}_{h}\|\right ).
$$

By \eqref{phi-phi} and
\begin{align*}
\left | (\mathrm{d}_{h}^{\ast} {u}_{h}, {\phi}_{h}-{\mathcal{H}_0^{k-1}}{{\phi}_{h}}) \right |&\lesssim h^{s}\|{u}_{h}-\Pi_{k}{\mathcal{H}^k}{{u}_{h}}\|\|\mathrm{d}_{h}^{\ast}{u}_{h}\|\\&
\leq h^{s}\left (\|{u}_{h}-{\mathcal{H}^k}{{u}_{h}}\|  + \|{\mathcal{H}^k}{{u}_{h}}-\Pi_{k}{\mathcal{H}^k}{{u}_{h}}\|   \right)\|\mathrm{d}_{h}^{\ast} {u}_{h}\|\\&
\lesssim h^{s}\|{u}_{h}-{\mathcal{H}^k}{{u}_{h}}\|\|\mathrm{d}_{h}^{\ast}{u}_{h}\|  +  h^{2s}\|\mathrm{d}_{h}^{\ast} {u}_{h}\|^{2}+  h^{2s}\|\mathrm{d} {u}_{h}\|^{2}\\&
\leq \frac{1}{2} \|{u}_{h}-{\mathcal{H}^k}{{u}_{h}}\|^{2}+\frac{1}{2}h^{2s}\|\mathrm{d}_{h}^{\ast} {u}_{h}\|^{2}  +  h^{2s}\|\mathrm{d}_{h}^{\ast} {u}_{h}\|^{2}+  h^{2s}\|\mathrm{d} {u}_{h}\|^{2},
\end{align*}
we obtain
$$
\|{u}_{h}-{\mathcal{H}^k}{{u}_{h}}\|^{2}\lesssim
 \|{\mathcal{H}^k}{{u}_{h}}-\Pi_{k}{\mathcal{H}^k}{{u}_{h}}\|^{2}
+ h^{2s}\left  (  \|\mathrm{d}_{h}^{\ast} {u}_{h}\|^2 +\|\mathrm{d} {u}_{h}\|^2  \right ) .
$$

This completes the proof.
\end{proof}

\section{Vector proxies and applications}\label{sec:vector}

With the vector proxies \cite{Arnold.D;Falk.R;Winther.R.2010a}, the generalized Gaffney inequalities for discrete differential forms yield estimates for the finite element methods. Some of these estimates are, as far as we know, new. 

The Hodge Laplacian problems in three space dimensions boils down to the Poisson equation
$$
-\Delta u =f,
$$
and the vector Laplacian problem
$$
\curl\curl \bm{w}-\grad\div \bm{w}=\bm{g}, 
$$
respectively.  Let $\Omega$ be an $s$-regular domain and $p=3/(3-s)$, and let $\grad_{h}$, $\curl_{h}$, $\div_{h}$ be the $L^{2}$ adjoint operators of $-\div_{h}$, $\curl_{h}$, $-\grad_{h}$ respectively. Below we use the generalized Gaffney inequality to give some estimates for various finite element discretizations for these two problems. 

\paragraph{Primal formulation for the scalar Poisson.} In this case we have 
\begin{align}\label{scalar-poisson}
-\div_{h}\grad u_{h}=\mathbb{P}_{0}f,
\end{align}
where $u_{h}$ is discretized by the Lagrange elements and $\mathbb{P}_{0}$ is the $L^{2}$ projection to the finite element space. The energy estimate gives $\|\grad u_{h}\|\leq \|f\|$. Then the standard Poincar\'{e} inequality and the Sobolev imbedding  imply $\|u_{h}\|_{0, p}\lesssim \|f\|$.

Considering $u_{h}$ as a discrete 1-form,  we conclude from Theorem \ref{thm:discrete-sobolev} with $k=1$, the equation \eqref{scalar-poisson} and the identity $\curl\grad u_{h}=0$ that the inequality $\|\grad u_h\|_{0, p}\lesssim \|f\|$ holds.

\paragraph{Mixed formulation for the scalar Possion.}

The mixed finite element formulation for the scalar Poisson equation boils down to solving 
\begin{align}\label{divgrad-h}
-\div\grad_{h}u_{h}=\mathbb{P}_{3}f,
\end{align}
 where $u_{h}$ is discretized by piecewise polynomials identified as a discrete 3-form and $\mathbb{P}_{3}$ is the $L^{2}$ projection to this space. In the implementation, one more variable $\bm{\sigma}_{h}=\grad_{h}u_{h}$ in the BDM/RT space is introduced.  

Testing the equation by $u_{h}$, we get $\|\grad u_{h}\|\leq \|f\|$. Together with proper boundary conditions, Theorem \ref{thm:discrete-sobolev} with $k=3$ let us conclude with the estimate $\|u_{h}\|_{0, p}\lesssim \|f\|$.   Considering $\bm{\sigma}_{h}=\grad_{h}u_{h}$ as a discrete 2-form, we further get from \eqref{divgrad-h} and the identity $\curl_{h}\grad_{h}u_{h}=0$ that
$\|\bm{\sigma}_{h}\|_{0, p}\lesssim \|f\|$.

\paragraph{1-form based mixed formulation for the vector Laplacian.}

Treating $\bm{w}_{h}$ as a discrete 1-form, we obtain a mixed finite element discretization for the vector Laplacian problem:
\begin{align}\label{1-form}
\curl_{h}\curl \bm{w}_{h}-\grad\div_{h}\bm{w}_{h}=\mathbb{P}_{1}\bm{g},
\end{align}
where $\bm{w}_{h}$ is discretized by the 1st/2nd N\'{e}d\'{e}lec element and $\mathbb{P}_{1}$ is the corresponding $L^{2}$ projection .

Testing \eqref{1-form} by $\bm{w}_{h}$, one obtains the estimate $\|\curl \bm{w}_{h}\|+\|\div_{h} \bm{w}_{h}\|\leq \|\bm{g}\|$. Then Theorem \ref{thm:discrete-sobolev} with $k=1$ implies $\|\bm{w}_{h}\|_{0, p}\lesssim \|\bm{g}\|$.

\paragraph{2-form based mixed formulation for the vector Laplacian.}

Discretizing $\bm{w}_{h}$ as a discrete 2-form in the BDM/RT space, we get another mixed finite element method for the vector Laplacian:
$$
\curl\curl_{h} \bm{w}_{h}-\grad_{h}\div\bm{w}_{h}=\mathbb{P}_{2}\bm{g},
$$
where $\mathbb{P}_{2}$ is the $L^{2}$ projection to the finite element space. In this case, we have $\|\curl_{h} \bm{w}_{h}\|+\|\div \bm{w}_{h}\|\leq \|\bm{g}\|$ and Theorem \ref{thm:discrete-sobolev} with $k=2$ let us conclude that $\|\bm{w}_{h}\|_{0, p}\lesssim \|\bm{g}\|$.

\section{Conclusion}\label{sec:conclusion}

We generalize the Hodge mapping for the weakly divergence-free edge elements \cite{Hiptmair.R.2002a} and the strongly divergence-free face elements \cite{hu2015structure} to  general  discrete differential forms.  Based on the new Hodge mapping, we further  prove the generalized Gaffney inequality and the discrete compactness for discrete differential forms.  

In the study of the Hodge mappings, the commuting interpolations act as a bridge between the continuous and the discrete levels. Therefore we hope that the techniques presented in this paper could be further explored for high order methods ($p$- version) or problems involving  general Hilbert complexes \cite{Arnold.D;Falk.R;Winther.R.2010a} provided that we have regularity results at the continuous level and suitable bounded commuting interpolations.

Several variants of the discrete compactness exist (c.f. \cite{christiansen2012variational}). The results in this paper may be further explored for these variants.

\section*{Acknowledgement}

Juncai He is supported in part by National Natural Science Foundation of China (NSFC) (Grant No. 91430215) and the Elite Program of Computational and Applied Mathematics for PHD Candidates of Peking University.
Kaibo Hu is supported in part by the  European Research Council under the European Union's Seventh Framework Programme (FP7/2007-2013) / ERC grant agreement 339643.
Jinchao Xu is supported in part by DOE Grant DE-SC0014400.

The authors are grateful to Prof. Ralf Hiptmair for several helpful suggestions.
\bibliographystyle{plain}      
\bibliography{gaffney}{}   

\end{document}